\documentclass[a4paper]{article}%
\usepackage[dvips]{graphicx}
\usepackage{amsmath}
\usepackage{amsfonts}
\usepackage{amssymb}
\usepackage{latexsym}
\usepackage{makeidx}
\usepackage[noBBpl]{mathpazo}
\usepackage{indentfirst}%
\providecommand{\U}[1]{\protect\rule{.1in}{.1in}}
\newtheorem{theorem}{Theorem}

\newtheorem{corollary}[theorem]{Corollary}

\newtheorem{lemma}[theorem]{Lemma}

\newtheorem{proposition}[theorem]{Proposition}
\newtheorem{remark}[theorem]{Remark}

\newenvironment{proof}[1][Proof]{\noindent\textbf{#1.} }{$\hfill\Box$\vspace*{.8cm}}
\begin{document}

\title{On the resonant Lane-Emden problem for the $p$-Laplacian}
\author{Grey Ercole{\small \textbf{ }}\thanks{The author thanks the support of FAPEMIG
and CNPq, Brazil.}\\{\small \textit{Departamento de Matem\'{a}tica - ICEx, Universidade Federal de
Minas Gerais,}}\\{\small \textit{Av. Ant\^{o}nio Carlos 6627, Caixa Postal 702, 30161-970, Belo
Horizonte, MG, Brazil. E-mail: grey@mat.ufmg.br }} }
\maketitle

\begin{abstract}
\noindent We study the positive solutions of the Lane-Emden equation
$-\Delta_{p}u=\lambda_{p}\left\vert u\right\vert ^{q-2}u$ in $\Omega$ with
homogeneous Dirichlet boundary conditions, where $\Omega\subset\mathbb{R}^{N}$
is a bounded and smooth domain, $N\geq2,$ $\lambda_{p}$ is the first
eigenvalue of the $p$-Laplacian operator $\Delta_{p}$ and $q$ is close to
$p>1.$ We prove that any family of positive solutions of this problem
converges in $C^{1}(\overline{\Omega})$ to the function $\theta_{p}e_{p}$ when
$q\rightarrow p,$ where $e_{p}$ is the positive and $L^{\infty}$-normalized
first eigenfunction of the $p$-Laplacian and $\theta_{p}:=\exp\left(
\left\Vert e_{p}\right\Vert _{L^{p}(\Omega)}^{-p}\int_{\Omega}e_{p}%
^{p}\left\vert \ln e_{p}\right\vert dx\right)  .$ A consequence of this result
is that the best constant of the immersion $W_{0}^{1,p}(\Omega)\hookrightarrow
L^{q}(\Omega)$ is differentiable at $q=p.$ Previous results on the asymptotic
behavior (as $q\rightarrow p$) of the positive solutions of the non-resonant
Lane-Emden problem (i.e. with $\lambda_{p}$ replaced by a positive
$\lambda\neq\lambda_{p}$) are also generalized to the space $C^{1}%
(\overline{\Omega})$ and to arbitrary families of these solutions. Moreover,
if $u_{\lambda,q}$ denotes a solution of the non-resonant problem for an
arbitrarily fixed $\lambda>0,$ we show how to obtain the first eigenpair of
the $p$-Laplacian as the limit in $C^{1}(\overline{\Omega}),$ when
$q\rightarrow p$, of a suitable scaling of the pair $(\lambda,u_{\lambda,q}).$
For computational purposes the advantage of this approach is that $\lambda$
does not need to be close to $\lambda_{p}.$ Finally, an explicit estimate
involving $L^{\infty}$ and $L^{1}$ norms of $u_{\lambda,q}$ is also deduced
using set level techniques.\newline 


\end{abstract}

\section{Introduction}

Consider the Lane-Emden problem
\begin{equation}
\left\{
\begin{array}
[c]{rcll}%
-\Delta_{p}u & = & \lambda\left\vert u\right\vert ^{q-2}u & \text{in }%
\Omega,\\
u & = & 0 & \text{on }\partial\Omega,
\end{array}
\right.  \label{Lane-Emden}%
\end{equation}
where $\lambda>0,$ $\Omega\subset\mathbb{R}^{N}$ is a bounded and smooth
domain, $N\geq2,$ $\Delta_{p}u:=\operatorname{div}\left(  \left\vert \nabla
u\right\vert ^{p-2}\nabla u\right)  $ is the $p$-Laplacian operator with
$p>1,$ and $1<q<p^{\ast},$ with $p^{\ast}$ denoting the Sobolev critical
exponent defined by $p^{\ast}=Np/\left(  N-p\right)  ,$ if $1<p<N,$ and
$p^{\ast}=\infty,$ if $p\geqslant N.$

If $q=p$, we have the $p$-Laplacian eigenvalue problem
\begin{equation}
\left\{
\begin{array}
[c]{rcll}%
-\Delta_{p}u & = & \lambda\left\vert u\right\vert ^{p-2}u & \text{in }\Omega\\
u & = & 0 & \text{on }\partial\Omega
\end{array}
\right.  \label{peigen}%
\end{equation}
whose first eigenvalue $\lambda_{p}$ is positive, simple, isolated and admits
a first positive eigenfunction $e_{p}\in C^{1,\alpha}\left(  \overline{\Omega
}\right)  $ satisfying $\left\Vert e_{p}\right\Vert _{\infty}=1$ in $\Omega.$
(We maintain this notation from now on.) Moreover, $\lambda_{p}$ also is
characterized by the minimizing property%
\begin{equation}
\lambda_{p}=\min\left\{  \frac{\int_{\Omega}\left\vert \nabla u\right\vert
^{p}dx}{\int_{\Omega}\left\vert u\right\vert ^{p}dx}:\text{ \ }u\in
W_{0}^{1,p}\left(  \Omega\right)  \setminus\{0\}\right\}  =\frac{\int_{\Omega
}\left\vert \nabla e_{p}\right\vert ^{p}dx}{\int_{\Omega}\left\vert
e_{p}\right\vert ^{p}dx}. \label{lambdamin}%
\end{equation}

We recall that $u\in W_{0}^{1,p}\left(  \Omega\right)  $ is a weak solution of
(\ref{Lane-Emden}) if, and only if,
\begin{equation}
\int_{\Omega}\left\vert \nabla u\right\vert ^{p-2}\nabla u\cdot\nabla\varphi
dx=\lambda\int_{\Omega}\left\vert u\right\vert ^{q-2}u\varphi dx\text{ \ for
all }\varphi\in W_{0}^{1,p}\left(  \Omega\right)  . \label{qweak}%
\end{equation}
This means that $u$ is a critical point of the energy functional
$J_{\lambda,q}:W_{0}^{1,p}(\Omega)\rightarrow\mathbb{R}$ given by%
\[
J_{\lambda,q}\left(  v\right)  =\frac{1}{p}\int_{\Omega}\left\vert \nabla
v\right\vert ^{p}dx-\frac{\lambda}{q}\int_{\Omega}\left\vert v\right\vert
^{q}dx.
\]

In the super-linear case $p<q<p^{\ast}$ the existence of at least one positive
weak solution $u_{\lambda,q}$ of (\ref{Lane-Emden}) with the least energy
$J_{\lambda,q}$ among all possible nontrivial weak solutions follows from
standard variational methods. Weak positive solutions satisfying this
minimizing property are known as ground states. As shown in
\cite{Garcia-Peral} non-uniqueness of positive weak solutions occurs for
ring-shaped domains when $q$ is close to $p^{\ast}$ (see also \cite{Drabek}).
On the other hand, when $\Omega$ is a ball, there exists only one positive
weak solution (see \cite{AY}). For the Laplacian ($p=2$) and a general bounded
domain, uniqueness happens if $q$ is sufficiently close to $2$ (see
\cite[Lemma 1]{Dancer}).

In the sub-linear case $1<q<p$ the existence of a positive weak solution
follows from the sub- and super-solution method or from standard variational
arguments concerning the global minimum of the energy functional
$J_{\lambda,q}$ in $W_{0}^{1,p}(\Omega).$ The uniqueness of such a weak
positive solution follows from \cite{DiazSaa} where a more general result is proved.

In both cases a proof of existence by applying the subdifferential method can
be found in \cite{Otani}, where one can also find the proof of the boundedness
(in the sup norm) of any positive weak solution of (\ref{Lane-Emden}), a
result which implies the $C^{1,\alpha}(\overline{\Omega})$-regularity by
applying well-known estimates (see \cite{DiBenedetto, Lieberman, Tolks}).

With different goals, asymptotics of solutions of the Lane-Emdem problem
(\ref{Lane-Emden}) has been studied by many authors since the 1990s. For
example, in \cite{Garcia-Peral} for $p>N,$ $\lambda=1$ and $q\rightarrow
p^{\ast}$; or in \cite{Wei} for $p=N,$ $\lambda=1$ and $q\rightarrow\infty$.
In \cite{GP}, the asymptotic behavior in $W_{0}^{1,p}(\Omega)$ of the positive
ground state solutions $v_{\lambda,q},$ as $q\rightarrow p^{+},$ was described
for all positive values of $\lambda.$ In that paper $v_{\lambda,q}$ was
obtained as the minimum of $J_{\lambda,q}$ on the positive Nehari manifold.
More recently, the asymptotic behavior with $q\rightarrow p^{-}$ in
$W_{0}^{1,p}(\Omega)$ was described in \cite{Anello}. Some these asymptotics
on the non-resonant problem (that is, $0<\lambda\neq\lambda_{p}$) had already
appeared in \cite{Huang}.

However, up to our knowledge, only in \cite{GP} and \cite{Anello} the resonant
problem was dealt with, but the asymptotic behavior of its positive solutions
was not fully determined. Indeed, although the families of solutions were
known to have a subsequence converging in $W_{0}^{1,p}(\Omega)$ to a multiple
of $e_{p},$ this multiple was unknown; in principle, different multiples of
$e_{p}$ could be obtained as limits of different subsequences these families.
Moreover, in the super-linear case, except for ground state families, nothing
was known about the asymptotic behavior (as $q\rightarrow p^{+}$) of other
(eventually existing) families of positive solutions.

In the present paper we first consider the resonant Lane-Emdem problem
\begin{equation}
\left\{
\begin{array}
[c]{rcll}%
-\Delta_{p}u & = & \lambda_{p}\left\vert u\right\vert ^{q-2}u & \text{in
}\Omega,\\
u & = & 0 & \text{on }\partial\Omega,
\end{array}
\right.  \label{Resonant}%
\end{equation}
and an arbitrary family $\left\{  u_{q}\right\}  _{q\in\lbrack1,p)\cup
(p,p^{\ast})}$ of positive solutions (not necessarily ground states, in the
super-linear case). By using Picone's inequality and blow-up arguments we
prove our main result: the convergence $u_{q}\rightarrow\theta_{p}e_{p}$ in
$C^{1}(\overline{\Omega}),$ as $q\rightarrow p,$ where%
\[
\theta_{p}:=\exp\left(  \left\Vert e_{p}\right\Vert _{p}^{-p}\int_{\Omega
}e_{p}^{p}\left\vert \ln e_{p}\right\vert dx\right)  .
\]

As a consequence, we obtain the differentiability at $q=p$ of the function
$q\in\lbrack1,p^{\ast})\mapsto\lambda_{q}\in\mathbb{R}$, where $\lambda_{q}$
denotes the minimum on $W_{0}^{1,p}(\Omega)\backslash\{0\}$ of the Rayleigh
quotient $\mathcal{R}_{q}$ defined by $\mathcal{R}_{q}(u):=\left\Vert \nabla
u\right\Vert _{p}^{p}/\left\Vert u\right\Vert _{q}^{p}.$ (From now on
$\left\Vert v\right\Vert _{r}$ stands for the usual $L^{r}$ norm of $v.$)
Precisely, we prove that
\[
\frac{d}{dq}\left[  \lambda_{q}\right]  _{q=p}=\lambda_{p}\ln(\theta
_{p}\left\Vert e_{p}\right\Vert _{p}).
\]
For this we use the fact that the function $v_{q}:=\left(  \frac{\lambda_{q}%
}{\lambda_{p}}\right)  ^{\frac{1}{q-p}}w_{q}$ is a positive weak solution of
the resonant Lane-Emden problem (\ref{lambdap}) for each $q\in\lbrack
1,p)\cup(p,p^{\ast})$, where $w_{q}$ denotes a positive extrema of the
Rayleigh quotient $\mathcal{R}_{q}.$ In the super-linear case $p<q<p^{\ast}$
the function $v_{q}$ is a ground state and in the sub-linear case $1<\,q<p$
this function is, of course, the only positive solution of (\ref{lambdap}).

We emphasize that our results determine the exact asymptotic behavior of
positive solutions, as $q\rightarrow p,$ of the Lane-Emden problem
(\ref{Lane-Emden}) for any $\lambda>0.$ In fact, any family of positive
solutions $\left\{  u_{\lambda,q}\right\}  _{q\in\lbrack1,p)\cup(p,p^{\ast})}$
of this problem in the non-resonant case $0<\lambda\neq\lambda_{p}$ is
obtained, by scaling, from a family of positive solutions of the resonant
case. Thus, as we will see, our main result implies that
\begin{equation}
\lim\limits_{q\rightarrow p^{-}}\left\Vert u_{\lambda,q}\right\Vert _{C^{1}%
}=\left\{
\begin{array}
[c]{ccc}%
0, & \text{if} & \lambda<\lambda_{p}\\
\infty, & \text{if} & \lambda>\lambda_{p}%
\end{array}
\right.  \qquad\text{and}\qquad\lim\limits_{q\rightarrow p^{+}}\left\Vert
u_{\lambda,q}\right\Vert _{C^{1}}=\left\{
\begin{array}
[c]{ccc}%
\infty, & \text{if} & \lambda<\lambda_{p}\\
0, & \text{if} & \lambda>\lambda_{p}%
\end{array}
\right.  \label{asympnon}%
\end{equation}
(Here $\left\Vert v\right\Vert _{C^{1}}:=\left\Vert v\right\Vert _{\infty
}+\left\Vert \nabla v\right\Vert _{\infty}$ is the norm of a function $v$ in
$C^{1}(\overline{\Omega}).$)

A third consequence of our main result is that, for each $\lambda>0,$ $1\leq
s\leq\infty$ and for any sequence $q_{n}\rightarrow p$ one has:
\[
\lim_{q_{n}\rightarrow p}\left(  \lambda\left\Vert u_{\lambda,q_{n}%
}\right\Vert _{s}^{q_{n}-p}\right)  =\lambda_{p}\qquad\text{and}\qquad
\frac{u_{\lambda,q_{n}}}{\left\Vert u_{\lambda,q_{n}}\right\Vert _{s}%
}\rightarrow\frac{e_{p}}{\left\Vert e_{p}\right\Vert _{s}}%
\]
the last convergence being in the $C^{1}(\overline{\Omega})$ space.

This result might be useful for numerical computation of the first eigenvalue
of the $p$-Laplacian (see \cite{BBEM}) taking into account the following
aspects: a) $\lambda$ does not need to be close to $\lambda_{p};$ b) the
sequence $q_{n}$ tending to $p$ can be arbitrarily chosen ; c) the
normalization in the computational processes can be made by using any $L^{s}%
$-norm with $s\geq1.$

Finally, by using level set techniques we prove an explicit estimate involving
$L^{\infty}$ and $L^{1}$ norms of the solutions $u_{\lambda,q}$ of
(\ref{Lane-Emden}), which is valid if $q\in\lbrack1,p\frac{N+1}{N}).$ This
estimate, which has independent interest, might be useful in a computational
approach of the Lane-Emden problem or even in the analysis of nodal solutions
for this problem.

This paper is organized as follows. In Section 2 we prove our main result on
the asymptotic behavior for the resonant case, as $q\rightarrow p$. Section 3
is dedicated to the consequences of our main result. Finally, in Section 4, we
obtain, for each $\lambda>0,$ an estimate involving $L^{\infty}$ and $L^{1}$
norms of the solution $u_{\lambda,q}.$

\section{Asymptotic behavior of the resonant problem\label{section ground}}

In this section we consider the resonant Lane-Emden problem%
\begin{equation}
\left\{
\begin{array}
[c]{rcll}%
-\Delta_{p}u & = & \lambda_{p}\left\vert u\right\vert ^{q-2}u & \text{in
}\Omega\\
u & = & 0 & \text{on }\partial\Omega.
\end{array}
\right.  \label{lambdap}%
\end{equation}
Our goal is to completely determine the asymptotic behavior of the weak
positive solutions of this problem, as $q\rightarrow p.$

The weak solutions of (\ref{lambdap}) are the critical points of the energy
functional $I_{q}:W_{0}^{1,p}\left(  \Omega\right)  \longrightarrow\mathbb{R}$
defined by
\[
I_{q}(u):=\frac{1}{p}\int_{\Omega}\left\vert \nabla u\right\vert ^{p}%
dx-\frac{\lambda_{p}}{q}\int_{\Omega}\left\vert u\right\vert ^{q}dx.
\]
Furthermore, a family $\left\{  v_{q}\right\}  _{q\in\lbrack1,p)\cup
(p,p^{\ast})}$ of positive weak solutions for (\ref{lambdap}) can be obtained
from minimizers of the Rayleigh quotient%
\[
\mathcal{R}_{q}(u):=\frac{\int_{\Omega}\left\vert \nabla u\right\vert ^{p}%
dx}{\left(  \int_{\Omega}\left\vert u\right\vert ^{q}dx\right)  ^{\frac{p}{q}%
}}%
\]
in $W_{0}^{1,p}\left(  \Omega\right)  \setminus\{0\}.$

In fact, as it is well-known, the compactness of the immersion $W_{0}%
^{1,p}(\Omega)\hookrightarrow L^{q}(\Omega)$ for $1\leq q<p^{\ast}$ implies
that $\mathcal{R}_{q}:W_{0}^{1,p}\left(  \Omega\right)  \setminus
\{0\}\longrightarrow\mathbb{R}$ attains a positive minimum at a positive and
$L^{q}$-normalized function $w_{q}\in W_{0}^{1,p}\left(  \Omega\right)  \cap
C^{1,\alpha}\left(  \overline{\Omega}\right)  :$
\begin{equation}
\left\Vert w_{q}\right\Vert _{q}=1\text{ \ and \ }\lambda_{q}:=\min\left\{
\mathcal{R}_{q}(u):u\in W_{0}^{1,p}\left(  \Omega\right)  \setminus
\{0\}\right\}  =\mathcal{R}_{q}(w_{q}). \label{cpinf}%
\end{equation}

It is straightforward to verify that $w_{q}$ is a weak solution of
\[
\left\{
\begin{array}
[c]{rcll}%
-\Delta_{p}u & = & \lambda_{q}\left\vert u\right\vert ^{q-2}u & \text{in
}\Omega\\
u & = & 0 & \text{on }\partial\Omega
\end{array}
\right.
\]
and hence that
\begin{equation}
v_{q}=\left(  \frac{\lambda_{q}}{\lambda_{p}}\right)  ^{\frac{1}{q-p}}w_{q}
\label{uq}%
\end{equation}
is a positive weak solution of (\ref{lambdap}) for each $q\in\lbrack
1,p)\cup(p,p^{\ast}).$

Since $\left\Vert w_{q}\right\Vert _{q}=1$ one has
\begin{equation}
\left\Vert v_{q}\right\Vert _{q}=\left(  \frac{\lambda_{q}}{\lambda_{p}%
}\right)  ^{\frac{1}{q-p}}. \label{lquq}%
\end{equation}

In the sub-linear case $1\leq q<p$ the function $v_{q}$ is the only critical
point of $I_{q}.$ Moreover, this function minimizes the energy functional
$I_{q}$ on $W_{0}^{1,p}\left(  \Omega\right)  \setminus\{0\},$ that is%
\begin{equation}
I_{q}(v_{q})=\min\left\{  I_{q}(v):v\in W_{0}^{1,p}\left(  \Omega\right)
\setminus\{0\}\right\}  . \label{minsub}%
\end{equation}
This property can also be directly proved using (\ref{cpinf}) and
(\ref{lquq}). In fact, it is straightforward to verify that if $v\in
W_{0}^{1,p}\left(  \Omega\right)  \setminus\{0\}$ then%
\[
I_{q}(v)\geq\min_{t\in\mathbb{R}}I_{q}(tv)=I_{q}(t_{v}v)\geq I_{q}(v_{q})
\]
where $t_{v}=\left(  \lambda_{p}\int_{\Omega}\left\vert v\right\vert
^{q}dx\right)  ^{\frac{1}{p-q}}\left(  \int_{\Omega}\left\vert \nabla
v\right\vert ^{p}dx\right)  ^{-\frac{1}{p-q}}.$

In the super-linear case $1<p<q<p^{\ast}$ the energy functional is not bounded
from below. Indeed, for any $v\in W_{0}^{1,p}\left(  \Omega\right)
\setminus\{0\}$ one can verify that%
\[
\lim_{t\rightarrow\infty}I_{q}(tv)=-\infty.
\]
However, the weak positive solution $v_{q}$ minimizes the energy functional
$I_{q}$ in the Nehari manifold%
\[
\mathcal{N}_{q}:=\left\{  v\in W_{0}^{1,p}\left(  \Omega\right)
\setminus\{0\}:\int_{\Omega}\left\vert \nabla v\right\vert ^{p}dx=\lambda
_{p}\int_{\Omega}\left\vert v\right\vert ^{q}dx\right\}  .
\]
Therefore, since all nontrivial solutions of (\ref{lambdap}) belong to
$\mathcal{N}_{q}$ (take $\lambda=\lambda_{p}$ and $\phi=u$ in (\ref{qweak})),
it follows that $v_{q}\in\mathcal{N}_{q}$ and also that $v_{q}$ is a ground state.

The verification that $v_{q}$ minimizes the energy functional in the Nehari
manifold $\mathcal{N}_{q}$ is simple: if $v\in\mathcal{N}_{q}$ then
\[
\lambda_{q}\leq\mathcal{R}_{q}(v)=\frac{\int_{\Omega}\left\vert \nabla
v\right\vert ^{p}dx}{\left(  \int_{\Omega}\left\vert v\right\vert
^{q}dx\right)  ^{\frac{p}{q}}}=\frac{\lambda_{p}\int_{\Omega}\left\vert
v\right\vert ^{q}dx}{\left(  \int_{\Omega}\left\vert v\right\vert
^{q}dx\right)  ^{\frac{p}{q}}}=\lambda_{p}\left\Vert v\right\Vert _{q}^{q-p},
\]
implying that
\begin{equation}
\left\Vert v_{q}\right\Vert _{q}=\left(  \frac{\lambda_{q}}{\lambda_{p}%
}\right)  ^{\frac{1}{q-p}}\leq\left\Vert v\right\Vert _{q}. \label{minlq}%
\end{equation}
Therefore,
\begin{equation}
I_{q}(v_{q})\leq I_{q}(v) \label{miner}%
\end{equation}
since
\begin{align*}
I_{q}(v_{q})  &  =\frac{1}{p}\int_{\Omega}\left\vert \nabla v_{q}\right\vert
^{p}dx-\frac{\lambda_{p}}{q}\int_{\Omega}\left\vert v_{q}\right\vert ^{q}dx\\
&  =\lambda_{p}\left(  \frac{1}{p}-\frac{1}{q}\right)  \left\Vert
v_{q}\right\Vert _{q}^{q}\\
&  \leq\lambda_{p}\left(  \frac{1}{p}-\frac{1}{q}\right)  \left\Vert
v\right\Vert _{q}^{q}=\frac{1}{p}\int_{\Omega}\left\vert \nabla v\right\vert
^{p}dx-\frac{\lambda_{p}}{q}\int_{\Omega}\left\vert v\right\vert ^{q}%
dx=I_{q}\left(  v\right)  .
\end{align*}

\begin{remark}
Since no general uniqueness result is known for the super-linear case, the
existence of multiple ground states for (\ref{lambdap}) is possible, at least
in principle, for each fixed $q\in(p,p^{\ast}).$ However, all of them must
have the same energy and also the same $L^{q}$ norm. Moreover, if $u_{q}$ is
an arbitrary nontrivial weak solution of (\ref{lambdap}), then $\left\Vert
u_{q}\right\Vert _{q}\geq\left\Vert v_{q}\right\Vert _{q},$ according to
(\ref{minlq}).
\end{remark}

In the remaining of this section we denote by $v_{q}$ the function defined by
(\ref{uq}) and by $u_{q}$ any positive solution of the resonant Lane-Emden
(\ref{Resonant}). Obviously, in the sub-linear case we must have $u_{q}%
=v_{q}.$

\begin{lemma}
\label{A}Let $u_{q}\in W_{0}^{1,p}\left(  \Omega\right)  $ be a positive weak
solution of the resonant Lane-Emden (\ref{Resonant}) with $q\in\lbrack
1,p)\cup(p,p^{\ast}).$ Then,%
\[
\left\Vert u_{q}\right\Vert _{\infty}\geq A:=\left\{
\begin{array}
[c]{lll}%
\left\vert \Omega\right\vert ^{-1}\int_{\Omega}\left\vert e_{p}\right\vert
^{p}dx & \text{ if } & 1<q<p\\
1 & \text{ if } & 1<p<q<p^{\ast}.
\end{array}
\right.
\]

\end{lemma}

\begin{proof}
If $1\leq q<p$ uniqueness implies that $u_{q}=v_{q}.$ Hence, (\ref{minsub})
and the fact that $0<e_{p}\leq1$ in $\Omega$ yield%
\begin{align*}
\lambda_{p}\left(  \frac{1}{q}-\frac{1}{p}\right)  \int_{\Omega}\left\vert
u_{q}\right\vert ^{q}dx  &  =-I_{q}(u_{q})\\
&  \geq-I_{q}(e_{p})\\
&  =\lambda_{p}\int_{\Omega}\left(  \frac{\left\vert e_{p}\right\vert ^{q}}%
{q}-\frac{\left\vert e_{p}\right\vert ^{p}}{p}\right)  dx\geq\lambda
_{p}\left(  \frac{1}{q}-\frac{1}{p}\right)  \int_{\Omega}\left\vert
e_{p}\right\vert ^{p}dx.
\end{align*}
Therefore, since $\left\vert \Omega\right\vert ^{-1}\int_{\Omega}\left\vert
e_{p}\right\vert ^{p}dx\leq1,$ we have
\[
\left\vert \Omega\right\vert ^{-1}\int_{\Omega}\left\vert e_{p}\right\vert
^{p}dx\leq\left(  \left\vert \Omega\right\vert ^{-1}\int_{\Omega}\left\vert
e_{p}\right\vert ^{p}dx\right)  ^{\frac{1}{q}}\leq\left(  \left\vert
\Omega\right\vert ^{-1}\int_{\Omega}\left\vert u_{q}\right\vert ^{q}dx\right)
^{\frac{1}{q}}\leq\left\Vert u_{q}\right\Vert _{\infty}.
\]

If $1<p<q<p^{\ast}$ then $\left\Vert u_{q}\right\Vert _{\infty}\geq1$ because%
\[
\int_{\Omega}\left\vert u_{q}\right\vert ^{p}dx\leq\frac{1}{\lambda_{p}}%
\int_{\Omega}\left\vert \nabla u_{q}\right\vert ^{p}dx=\int_{\Omega}\left\vert
u_{q}\right\vert ^{q}dx\leq\left\Vert u_{q}\right\Vert _{\infty}^{q-p}%
\int_{\Omega}\left\vert u_{q}\right\vert ^{p}dx.
\]

\end{proof}

In the next lemma $\phi_{p}\in W_{0}^{1,p}(\Omega)$ denotes the $p$-torsion
function of $\Omega$, that is, the solution of
\begin{equation}
\left\{
\begin{array}
[c]{rcll}%
-\Delta_{p}u & = & 1 & \text{in }\Omega,\\
u & = & 0 & \text{on }\partial\Omega.
\end{array}
\right.  \label{torsion}%
\end{equation}
(Classical results imply that $\phi_{p}>0$ in $\Omega$ and that $\phi_{p}\in
C^{1,\beta}(\overline{\Omega})$ for some $0<\beta<1.$)

\begin{lemma}
\label{uppersub}For each $1\leq q<p,$ let $u_{q}\in W_{0}^{1,p}\left(
\Omega\right)  $ be the positive weak solution of the (sub-linear) Lane-Emden
(\ref{lambdap}). Then,
\begin{equation}
\left\Vert u_{q}\right\Vert _{\infty}^{p-q}\leq\lambda_{p}\left\Vert \phi
_{p}\right\Vert _{\infty}^{p-1}. \label{subupper}%
\end{equation}

\end{lemma}

\begin{proof}
Since
\[
\left\{
\begin{array}
[c]{ll}%
-\Delta_{p}u_{q}=\lambda_{p}u_{q}^{q-1}\leq\lambda_{p}\left\Vert
u_{q}\right\Vert _{\infty}^{q-1}=-\Delta_{p}\left(  \left(  \lambda
_{p}\left\Vert u_{q}\right\Vert _{\infty}^{q-1}\right)  ^{\frac{1}{p-1}}%
\phi_{p}\right)  & \text{in }\Omega,\\
u_{q}=0=\left(  \lambda_{p}\left\Vert u_{q}\right\Vert _{\infty}^{q-1}\right)
^{\frac{1}{p-1}}\phi_{p} & \text{on }\partial\Omega
\end{array}
\right.
\]
it follows from the comparison principle that $u_{q}\leq\left(  \lambda
_{p}\left\Vert u_{q}\right\Vert _{\infty}^{q-1}\right)  ^{\frac{1}{p-1}}%
\phi_{p}$ in $\Omega.$ Hence, we obtain (\ref{subupper}) after passing to the
maximum values.
\end{proof}

\begin{remark}
It follows from Lemmas \ref{A} and \ref{uppersub} that
\[
\frac{1}{\lambda_{p}\left\Vert \phi_{p}\right\Vert _{\infty}^{p-1}}\leq
\liminf_{q\rightarrow p^{-}}\left\Vert u_{q}\right\Vert _{\infty}^{q-p}%
\leq\limsup_{q\rightarrow p^{-}}\left\Vert u_{q}\right\Vert _{\infty}%
^{q-p}\leq1.
\]
which leads to following well-known lower bound to the first eigenvalue
$\lambda_{p}$ in terms of the $p$-torsion function of $\Omega:$%
\begin{equation}
\frac{1}{\left\Vert \phi_{p}\right\Vert _{\infty}^{p-1}}\leq\lambda_{p}.
\label{lblp}%
\end{equation}

\end{remark}

In the sequel we prove an \textit{a priori} $L^{\infty}$ boundedness result
for an arbitrary family $\left\{  u_{q}\right\}  _{p<q<p^{\ast}}$ of positive
weak solutions of the super-linear Lane-Endem problem (\ref{lambdap}), if $q$
is sufficiently close to $p^{+}.$ Our proof was motivated by Lemma 2.1 of
\cite{Lorca}, where a Liouville-type theorem was proved for positive weak
solutions of the inequality $-\Delta_{p}w\geq cw^{p-1}$ in $\mathbb{R}^{N}$ or
in a half-space. It combines a blow-up argument with the following Picone's
inequality (see \cite{AH}), which is valid for all differentiable $u\geq0$ and
$v>0:$
\begin{equation}
\left\vert \nabla u\right\vert ^{p}\geq\left\vert \nabla v\right\vert
^{p-2}\nabla v\cdot\nabla\left(  \frac{u^{p}}{v^{p-1}}\right)  .
\label{Picones}%
\end{equation}

\begin{lemma}
\label{uppersup}Let $\left\{  u_{q}\right\}  _{p<q<p^{\ast}}$ be a family of
positive weak solutions of the (super-linear) Lane-Endem problem
(\ref{lambdap}). Then,
\begin{equation}
\limsup_{q\rightarrow p^{+}}\left\Vert u_{q}\right\Vert _{\infty}^{q-p}%
<\infty. \label{blowbond}%
\end{equation}

\end{lemma}

\begin{proof}
Let us suppose, by contradiction, that $\left\Vert u_{q_{n}}\right\Vert
_{\infty}^{q_{n}-p}\rightarrow\infty$ for some sequence $q_{n}\rightarrow
p^{+}.$ Let $x_{n}$ denote a maximum point of $u_{q_{n}},$ so that $u_{q_{n}%
}(x_{n})=\left\Vert u_{q_{n}}\right\Vert _{\infty}.$ Define
\[
\mu_{n}:=\left(  \lambda_{p}\left\Vert u_{q_{n}}\right\Vert _{\infty}%
^{q_{n}-p}\right)  ^{-\frac{1}{p}},\text{ \ \ }\Omega_{n}:=\left\{
x\in\mathbb{R}^{N}:\mu_{n}x+x_{n}\in\Omega\right\}
\]
and
\[
w_{n}(x):=\left\Vert u_{q_{n}}\right\Vert _{\infty}^{-1}u_{q_{n}}(\mu
_{n}x+x_{n});\text{ \ }x\in\Omega_{n}.
\]

Note that $B_{d_{n}/\mu_{n}}\subset\Omega_{n}$ where we are denoting by
$d_{n}\ $the distance from $x_{n}$ to the boundary $\partial\Omega$ and by
$B_{d_{n}/\mu_{n}}$ the ball centered at $x=0$ with radius $d_{n}/\mu_{n}.$

It follows that $\mu_{n}\rightarrow0^{+},$ $0<w_{n}\leq1=\left\Vert
w_{n}\right\Vert _{\infty}=w_{n}(0)$ in $\Omega_{n}$ and
\begin{equation}
\left\{
\begin{array}
[c]{rcll}%
-\Delta_{p}w_{n} & = & w_{n}^{q_{n}-1} & \text{in }\Omega_{n},\\
w_{n} & = & 0 & \text{on }\partial\Omega_{n}.
\end{array}
\right.  \label{wn}%
\end{equation}

By passing to a subsequence we can also suppose that $x_{n}\rightarrow
x_{0}\in\overline{\Omega}$ and that $\left\Vert u_{q_{n}}\right\Vert _{\infty
}^{q_{n}-p}$ is increasing. It is well-known that $\Omega_{n}$ tends either to
$\mathbb{R}^{N}$ or to a half-space if $x_{0}\in\Omega$ or $x_{0}\in
\partial\Omega,$ respectively.

Let $B_{R}$ be a ball with radius $R$ sufficiently large satisfying
$0\in\overline{B_{R}}$ and
\begin{equation}
\lambda_{R}<1, \label{lambdaR<1}%
\end{equation}
where $\lambda_{R}$ denotes the first eigenvalue of the $p$-Laplacian with
homogeneous Dirichlet conditions in $B_{R}.$

Now, let $n_{0}$ be such that
\[
\overline{B_{R}}\subset\Omega_{n}\text{ \ for all }n\geq n_{0}.
\]

Since $0\leq w_{n}\leq1$ in $\overline{B_{R}},$ global H\"{o}lder regularity
implies that there exist constants $K>0$ and $\beta\in(0,1),$ both depending
at most on $R$ (but independent of $n\geq n_{0}$), such that $\left\Vert
w_{n}\right\Vert _{C^{1,\beta}(\overline{B_{R}})}\leq K$ (see \cite[Theorem
1]{Lieberman}). Hence, compactness of the immersion $C^{1,\beta}%
(\overline{B_{R}})\hookrightarrow C^{1}(\overline{B_{R}})$ implies that, up to
a subsequence, $w_{n}\rightarrow w$ in $C^{1}(\overline{B_{R}}).$ Note that
$w\geq0$ in $\overline{B_{R}}$ and $w(0)=1$ (since $w_{n}(0)=1$).

Moreover, we have
\[
-\Delta_{p}w=w^{p-1}\text{ \ in }B_{R}%
\]
in the weak sense. In fact, if $\phi\in C_{0}^{\infty}(B_{R})\subset
C_{0}^{\infty}(\Omega_{n})$ is an arbitrary test function of $B_{R}$ then
(\ref{wn}) yields
\[
\int_{B_{R}}\left\vert w_{n}\right\vert ^{p-2}\nabla w_{n}\cdot\nabla\phi
dx=\int_{B_{R}}w_{n}^{q_{n}-1}\phi dx.
\]
Thus, after making $n\rightarrow\infty$ we obtain%
\begin{equation}
\int_{B_{R}}\left\vert w\right\vert ^{p-2}\nabla w\cdot\nabla\phi
dx=\int_{B_{R}}w^{p-1}\phi dx. \label{weakBr}%
\end{equation}

We remark that the Strong Maximum Principle (see \cite{Vazquez}) really
implies that $w>0$ in $B_{R}.$

Now, let $e_{R}\in W_{0}^{1,p}(B_{R})\cap C^{1}(\overline{B_{R}})$ be a
positive first eingenfunction of the $p$-Laplacian for the ball $B_{R}.$ Since
$C_{0}^{\infty}(B_{R})$ is dense in $W_{0}^{1,p}(B_{R})$ the equality
(\ref{weakBr}) is also valid for all $\phi\in W_{0}^{1,p}(B_{R}).$ In
particular, it is valid for $\phi=e_{R}^{p}/w^{p-1}$ (see Remark \ref{lopital}
after this proof).

It follows from Picone's inequality that%
\begin{equation}
\int_{B_{R}}\left\vert \nabla e_{R}\right\vert ^{p}dx\geq\int_{B_{R}%
}\left\vert w\right\vert ^{p-2}\nabla w\cdot\nabla\left(  \frac{e_{R}^{p}%
}{w^{p-1}}\right)  dx. \label{Rpicone}%
\end{equation}
Hence, (\ref{weakBr}) yields%
\[
\lambda_{R}\int_{B_{R}}e_{R}^{p}dx\geq\int_{B_{R}}w^{p-1}\frac{e_{R}^{p}%
}{w^{p-1}}dx=\int_{B_{R}}e_{R}^{p}dx
\]
that is, $\lambda_{R}\geq1$ which contradicts (\ref{lambdaR<1}).
\end{proof}

\begin{remark}
\label{lopital}The quotient $e_{R}/w$ of $C^{1}$ functions is well-defined in
$\Omega$ (since $w>0$ there) and at the points of the boundary $\partial
B_{R}$ where $w$ is null. In fact, since $e_{R}=0$ on $\partial B_{R}$ this
fact is a consequence of the Hopf's Boundary Lemma (see \cite{Vazquez} again):
if $y\in\partial B_{R}$ is such that $w(y)=0$ then any inward directional
derivative of both $e_{R}$ and $w$ is positive. Thus, L'H\^{o}pital's rule
implies that
\[
\lim\limits_{\substack{x\rightarrow y\\x\in B_{R}}}\dfrac{e_{R}(x)}{w(x)}>0.
\]

\end{remark}

\begin{lemma}
\label{IntegUq}Let $\left\{  u_{q}\right\}  _{q\in\lbrack1,p)\cup(p,p^{\ast}%
)}$ be a family of positive solutions of the Lane-Endem problem (\ref{lambdap}%
) and define, for each $q\in\lbrack1,p)\cup(p,p^{\ast}),$ the function
$U_{q}:=\dfrac{u_{q}}{\left\Vert u_{q}\right\Vert _{\infty}}.$ Then $U_{q}$
converges to $e_{p}$ in $C^{1}(\overline{\Omega})$ as $q\rightarrow p.$
Moreover,
\begin{equation}%
{\displaystyle\int_{\Omega}}
\frac{U_{q}^{p}-U_{q}^{q}}{q-p}dx\rightarrow%
{\displaystyle\int_{\Omega}}
e_{p}^{p}\left\vert \ln e_{p}\right\vert dx\text{ \ \ as \ }q\rightarrow p.
\label{asy}%
\end{equation}

\end{lemma}

\begin{proof}
It is easy to verify that
\begin{equation}
\left\{
\begin{array}
[c]{rcll}%
-\Delta_{p}U_{q} & = & \lambda_{p}\left\Vert u_{q}\right\Vert _{\infty}%
^{q-p}U_{q}^{q-1} & \text{in }\Omega,\\
U_{q} & = & 0 & \text{on }\partial\Omega.
\end{array}
\right.  \label{DiriUq}%
\end{equation}
Thus, it follows from Lemmas \ref{A}, \ref{uppersub} and \ref{uppersup} that
\[
0<C_{1}\leq\left\Vert u_{q}\right\Vert _{\infty}^{q-p}\leq C_{2}%
\]
for all $q\in\lbrack1,p)\cup(p,p+\epsilon),$ for some $\epsilon>0,$ being
$C_{1}$ and $C_{2}$ constants that do not depend on $q\in\lbrack
1,p)\cup(p,p+\epsilon).$

Therefore, since the right-hand side of the equation in (\ref{DiriUq}) is
uniformly bounded with respect to $q\in\lbrack1,p)\cup(p,p+\epsilon)$ the
global H\"{o}lder regularity result again implies that $\left\Vert
U_{q}\right\Vert _{C^{1,\beta}(\overline{\Omega})}\leq K,$ where $K$ and
$0<\beta<1$ are also uniform with respect to $q\in\lbrack1,p+\epsilon].$

Hence, compactness of the immersion $C^{1,\alpha}(\overline{\Omega
})\hookrightarrow C^{1}(\overline{\Omega})$ implies that, up to a subsequence,
$U_{q}$ converges in $C^{1}\left(  \overline{\Omega}\right)  $ to a function
$U\geq0$ (as $q\rightarrow p$) with $\left\Vert U\right\Vert _{\infty}=1.$ We
also have that $\lambda_{p}\left\Vert u_{q}\right\Vert _{\infty}%
^{q-p}\rightarrow c\in(\lambda_{p}C_{1},\lambda_{p}C_{2}).$

Taking the limit $q\rightarrow p$ in the weak formulation (\ref{qweak}) with
$\lambda=\lambda_{p}\left\Vert u_{q}\right\Vert _{\infty}^{q-p},$ we obtain%
\[
\int_{\Omega}\left\vert \nabla U\right\vert ^{p-2}\nabla U\cdot\nabla\varphi
dx=c\int_{\Omega}\left\vert U\right\vert ^{p-2}U\varphi dx
\]
for an arbitrary test function $\varphi\in W_{0}^{1,p}\left(  \Omega\right)
.$This proves that $U$ is a nonnegative eigenfunction associated with the
eigenvalue $c$ and such that $\left\Vert U\right\Vert _{\infty}=1.$ But this
fact necessarily implies that $c=\lambda_{p}$ and $U=e_{p}.$ Thus, the
uniqueness of the limits $\lambda_{p}\left\Vert u_{q}\right\Vert _{\infty
}^{q-p}\rightarrow\lambda_{p}$ and $U_{q}\rightarrow e_{p}$ show that these
convergences do not depend on subsequences. Therefore, we conclude that
$\left\Vert u_{q}\right\Vert _{\infty}^{q-p}\rightarrow1$ and that
$U_{q}\rightarrow e_{p}$ in $C^{1}\left(  \overline{\Omega}\right)  .$

In order to prove (\ref{asy}) we firstly observe that%
\[
\left\vert \frac{U_{q}^{p}-U_{q}^{q}}{q-p}\right\vert \leq\frac{1}{\left\vert
q-p\right\vert }\max_{0\leq t\leq1}\left\vert t^{p}-t^{q}\right\vert =\frac
{1}{\left\vert q-p\right\vert }\frac{1}{p}\left(  \frac{p}{q}\right)
^{\frac{q}{q-p}}\left\vert q-p\right\vert =\frac{1}{p}\left(  \frac{p}%
{q}\right)  ^{\frac{q}{q-p}},
\]
implying that $\dfrac{U_{q}^{p}-U_{q}^{q}}{q-p}$ is uniformly bounded with
respect to $q$ close to $p$ with%
\begin{equation}
\limsup_{q\rightarrow p}\left\vert \frac{U_{q}^{p}-U_{q}^{q}}{q-p}\right\vert
\leq\lim_{q\rightarrow p}\frac{1}{p}\left(  \frac{p}{q}\right)  ^{\frac
{q}{q-p}}=\frac{1}{p\exp(1)}. \label{UqBound}%
\end{equation}

Now, by taking into account the convergence $U_{q}\rightarrow e_{p}$ in
$C^{1}(\overline{\Omega}),$ (\ref{asy}) follows from Lebesgue's dominated
convergence theorem if we prove that
\[
\dfrac{1-U_{q}^{q-p}}{q-p}\rightarrow\left\vert \ln e_{p}\right\vert \text{
\ as }q\rightarrow p^{+}\text{ \ a.e. in }\Omega
\]
and%
\[
\dfrac{U_{q}^{p-q}-1}{q-p}\rightarrow\left\vert \ln e_{p}\right\vert \text{
\ as }q\rightarrow p^{-}\text{ \ a.e. in }\Omega.
\]

So, let $\mathcal{K}\subset\Omega$ compact and $0<\delta<\min
\limits_{\mathcal{K}}e_{p}.$ Then
\[
0<\min_{\mathcal{K}}e_{p}-\delta<e_{p}-\delta\leq U_{q}\leq e_{p}+\delta\text{
\ \ in }\mathcal{K}%
\]
for all $q$ sufficiently close to $p.$ Hence, in $\mathcal{K}$ one has
\begin{equation}
-\ln(e_{p}+\delta)\leq\liminf_{q\rightarrow p^{+}}\frac{1-U_{q}^{q-p}}%
{q-p}\leq\limsup_{q\rightarrow p^{+}}\frac{1-U_{q}^{q-p}}{q-p}\leq-\ln
(e_{p}-\delta), \label{asy3}%
\end{equation}
since
\[
\lim_{q\rightarrow p^{+}}\frac{1-(e_{p}+\delta)^{q-p}}{q-p}=-\ln(e_{p}%
+\delta)
\]
and%
\[
\lim_{q\rightarrow p^{+}}\frac{1-(e_{p}-\delta)^{q-p}}{q-p}=-\ln(e_{p}%
-\delta).
\]

Therefore, making $\delta\rightarrow0^{+}$ in (\ref{asy3}) we conclude that%
\[
\lim_{q\rightarrow p^{+}}\frac{1-U_{q}^{q-p}}{q-p}=-\ln e_{p}=\left\vert \ln
e_{p}\right\vert \text{ \ in }\mathcal{K}.
\]
Analogously we prove that%
\[
\lim_{q\rightarrow p^{-}}\dfrac{U_{q}^{p-q}-1}{q-p}=\left\vert \ln
e_{p}\right\vert \text{\ in }\mathcal{K}.
\]

\end{proof}

\begin{lemma}
\label{infsup1}Let $\left\{  u_{q}\right\}  _{q\in\lbrack1,p)\cup(p,p^{\ast}%
)}$ be a family of positive weak solutions of the Lane-Endem problem
(\ref{lambdap}). Then,%
\begin{equation}
\limsup_{q\rightarrow p^{-}}\left\Vert u_{q}\right\Vert _{\infty}\leq
\theta_{p}\leq\liminf_{q\rightarrow p^{+}}\left\Vert u_{q}\right\Vert
_{\infty}. \label{supinf1}%
\end{equation}

\end{lemma}

\begin{proof}
Let $U_{q}=\dfrac{u_{q}}{\left\Vert u_{q}\right\Vert _{\infty}}$ as in Lemma
\ref{IntegUq}. Applying Picone's inequality to $U_{q}$ and $e_{p}$ one has%
\begin{equation}
\int_{\Omega}\left\vert \nabla U_{q}\right\vert ^{p}dx\geq\int_{\Omega
}\left\vert \nabla e_{p}\right\vert ^{p-2}\nabla e_{p}\cdot\nabla\left(
\frac{U_{q}^{p}}{e_{p}^{p-1}}\right)  dx. \label{pic}%
\end{equation}
(Hopf's boundary lemma again implies that ${U}_{q}^{p}/e_{p}^{p-1}\in
W_{0}^{1,p}(\Omega).$) Therefore, it follows from (\ref{DiriUq}) that
\[
\lambda_{p}\left\Vert u_{q}\right\Vert _{\infty}^{q-p}\int_{\Omega}U_{q}%
^{q}dx\geq\lambda_{p}\int_{\Omega}e_{p}^{p-1}\frac{U_{q}^{p}}{e_{p}^{p-1}%
}dx=\lambda_{p}\int_{\Omega}U_{q}^{p}dx
\]
and from this we obtain%
\begin{equation}
\frac{\left\Vert u_{q}\right\Vert _{\infty}^{q-p}-1}{q-p}\int_{\Omega}%
U_{q}^{q}dx\geq\int_{\Omega}\frac{U_{q}^{p}-U_{q}^{q}}{q-p}dx\text{ \ if
\ }p<q<p^{\ast} \label{uq+}%
\end{equation}
and%
\begin{equation}
\frac{\left\Vert u_{q}\right\Vert _{\infty}^{q-p}-1}{q-p}\int_{\Omega}%
U_{q}^{q}dx\leq\int_{\Omega}\frac{U_{q}^{p}-U_{q}^{q}}{q-p}dx\text{ \ if
\ }1<q<p. \label{uq-}%
\end{equation}

\noindent\textbf{Case }$q\rightarrow p^{+}.$ Let us suppose, by contradiction,
that exist $L<\theta_{p}$ and a sequence $q_{n}\rightarrow p^{+}$ such that
$\left\Vert u_{q_{n}}\right\Vert _{\infty}\leq L.$ Then (\ref{uq+}) and Lemma
\ref{IntegUq} yield%
\[%
{\displaystyle\int_{\Omega}}
e_{p}^{p}\left\vert \ln e_{p}\right\vert dx=\lim\int_{\Omega}\frac{U_{q_{n}%
}^{p}-U_{q_{n}}^{q_{n}}}{q_{n}-p}dx\leq\lim\frac{L^{q_{n}-p}-1}{q_{n}-p}%
\int_{\Omega}U_{q_{n}}^{q_{n}}dx=\ln L%
{\displaystyle\int_{\Omega}}
e_{p}^{p}dx,
\]
that is, $\theta_{p}\leq L,$ thus reaching a contradiction. We have proved the
second inequality in (\ref{supinf1}).

\noindent\textbf{Case }$q\rightarrow p^{-}.$ Analogously, if we suppose that
exist $L>\theta_{p}$ and a sequence $q_{n}\rightarrow p^{-}$ such that
$\left\Vert u_{q_{n}}\right\Vert _{\infty}\geq L,$ then we obtain from
(\ref{uq+}) and Lemma \ref{IntegUq} that%
\[
\ln L%
{\displaystyle\int_{\Omega}}
e_{p}^{p}dx\leq\lim\frac{L^{q_{n}-p}-1}{q_{n}-p}\int_{\Omega}U_{q_{n}}^{q_{n}%
}dx\leq\lim\int_{\Omega}\frac{U_{q_{n}}^{p}-U_{q_{n}}^{q_{n}}}{q_{n}-p}dx=%
{\displaystyle\int_{\Omega}}
e_{p}^{p}\left\vert \ln e_{p}\right\vert dx
\]
and hence $L\leq\theta_{p}.$ This proves the first inequality in
(\ref{supinf1}).
\end{proof}

\begin{lemma}
\label{infsup2}Let $\left\{  u_{q}\right\}  _{q\in\lbrack1,p)\cup(p,p^{\ast}%
)}$ be a family of positive weak solutions of the Lane-Endem problem
(\ref{lambdap}). Then,
\begin{equation}
\limsup_{q\rightarrow p^{+}}\left\Vert u_{q}\right\Vert _{\infty}\leq
\theta_{p}\leq\liminf_{q\rightarrow p^{-}}\left\Vert u_{q}\right\Vert
_{\infty}. \label{supinf2}%
\end{equation}

\end{lemma}

\begin{proof}
By applying Picone's inequality again, but interchanging $U_{q}$ with $e_{q}$
in (\ref{pic}), the lemma follows similarly. In fact, we obtain
\begin{equation}
\frac{\left\Vert u_{q}\right\Vert _{\infty}^{q-p}-1}{q-p}\int_{\Omega}%
U_{q}^{q}(e_{p}/U_{q})^{p}dx\leq\int_{\Omega}\frac{U_{q}^{p}-U_{q}^{q}}%
{q-p}(e_{p}/U_{q})^{p}dx\text{ \ if \ }p<q<p^{\ast} \label{uq++}%
\end{equation}
and%
\begin{equation}
\frac{\left\Vert u_{q}\right\Vert _{\infty}^{q-p}-1}{q-p}\int_{\Omega}%
U_{q}^{q}(e_{p}/U_{q})^{p}dx\geq\int_{\Omega}\frac{U_{q}^{p}-U_{q}^{q}}%
{q-p}(e_{p}/U_{q})^{p}dx\text{ \ if \ }1\leq q<p. \label{uq--}%
\end{equation}
Note that the uniform convergence $U_{q}\rightarrow e_{p}$ in $\overline
{\Omega}$ together with the Hopf's boundary lemma guarantee that $e_{p}%
/U_{q}\rightarrow1$ uniformly in $\overline{\Omega}.$ Thus, it follows that
\[
\int_{\Omega}U_{q}^{q}(e_{p}/U_{q})^{p}dx\rightarrow\int_{\Omega}e_{p}%
^{p}dx,\text{ \ as \ }q\rightarrow p
\]
and%
\[
\int_{\Omega}\frac{U_{q}^{p}-U_{q}^{q}}{q-p}(e_{p}/U_{q})^{p}dx\rightarrow%
{\displaystyle\int_{\Omega}}
e_{p}^{p}\left\vert \ln e_{p}\right\vert dx,\text{ \ as \ }q\rightarrow p
\]
according to (\ref{asy}) and (\ref{UqBound}).

Thus, if we suppose that exist $L>\theta_{p}$ and $q_{n}\rightarrow p^{+}$
such that $\left\Vert u_{q_{n}}\right\Vert _{\infty}\geq L,$ the uniform
convergence $U_{q_{n}}\rightarrow e_{p}$ together with (\ref{uq++}) imply
that
\begin{equation}
\ln L=\lim\frac{L^{q_{n}-p}-1}{q_{n}-p}\leq\frac{\int_{\Omega}e_{p}%
^{p}\left\vert \ln e_{p}\right\vert dx}{\int_{\Omega}e_{p}^{p}dx}
\label{qasy2}%
\end{equation}
and hence we arrive at the contradiction $L\leq\theta_{p}.$ Therefore, the
first inequality in (\ref{supinf2}) holds.

On the other hand, if we assume, by contradiction again, the existence of
$L<\theta_{p}$ and $q_{n}\rightarrow p^{-}$ such that $\ \left\Vert u_{q_{n}%
}\right\Vert _{\infty}\leq L$ then it follows from (\ref{uq--}) that
\begin{align*}
\ln L\int_{\Omega}e_{p}^{p}dx  &  =\lim\frac{L^{q_{n}-p}-1}{q_{n}-p}%
\int_{\Omega}U_{q_{n}}^{q_{n}}(e_{p}/U_{q_{n}})^{p}dx\\
&  \geq\lim\int_{\Omega}\frac{U_{q_{n}}^{p}-U_{q_{n}}^{q_{n}}}{q_{n}-p}%
(e_{p}/U_{q_{n}})^{p}dx=%
{\displaystyle\int_{\Omega}}
e_{p}^{p}\left\vert \ln e_{p}\right\vert dx.
\end{align*}
Since this implies that $L\geq\theta_{p}$ we obtain a contradiction, proving
thus the second inequality in (\ref{supinf2}).
\end{proof}

It is worth to mention that in the Laplacian case $p=2$ the self-adjointness
of this operator produces
\[
\lim_{q\rightarrow2}\left\Vert u_{q}\right\Vert _{\infty}=\theta_{2}%
\]
directly. Such argument has already appeared in \cite{DancerDuMa}, where the
asymptotic behavior of positive solutions of a logistical type problem for the
Laplacian was studied. In fact,
\[
\lambda_{2}\left\Vert u_{q}\right\Vert _{\infty}^{q-2}\int_{\Omega}U_{q}%
^{q-1}e_{2}dx=\int_{\Omega}\nabla U_{q}\cdot\nabla e_{2}dx=\int_{\Omega}\nabla
e_{2}\cdot\nabla U_{q}dx=\lambda_{2}\int_{\Omega}e_{2}U_{q}dx
\]
leads to
\[
\frac{\left\Vert u_{q}\right\Vert _{\infty}^{q-2}-1}{q-2}\int_{\Omega}%
U_{q}^{q-1}e_{2}dx=\int_{\Omega}\frac{1-U_{q}^{q-2}}{q-2}e_{2}U_{q}dx.
\]
Thus, if $\left\Vert u_{q_{n}}\right\Vert _{\infty}\rightarrow L$ then
\[
\ln L=\lim_{n}\frac{\left\Vert u_{q_{n}}\right\Vert _{\infty}^{q_{n}-2}%
-1}{q_{n}-2}=\lim_{n}\frac{\int_{\Omega}\frac{1-U_{q_{n}}^{q_{n}-2}}{q_{n}%
-2}e_{2}U_{q_{n}}dx}{\int_{\Omega}U_{q_{n}}^{q_{n}-1}e_{2}dx}=\left\Vert
e_{2}\right\Vert _{2}^{-2}\int_{\Omega}e_{2}^{2}\left\vert \ln e_{2}%
\right\vert dx
\]
proving that $\left\Vert u_{q}\right\Vert _{\infty}\rightarrow\theta_{2}.$

\begin{theorem}
\label{mainTheo}Let $\left\{  u_{q}\right\}  _{q\in\lbrack1,p)\cup(p,p^{\ast
})}$ be a family of positive weak solutions of the Lane-Endem problem
(\ref{lambdap}). Then,
\[
\lim_{q\rightarrow p}u_{q}=\theta_{p}e_{p},
\]
the convergence being in $C^{1}(\overline{\Omega}).$
\end{theorem}

\begin{proof}
Lemmas \ref{infsup1} and \ref{infsup2} imply that
\begin{equation}
\lim_{q\rightarrow p}\left\Vert u_{p}\right\Vert _{\infty}\rightarrow
\theta_{p}. \label{limuq}%
\end{equation}
Thus, the right-hand side of (\ref{lambdap}) is bounded for all $q$
sufficiently close to $p.$ This fact, combined with the global H\"{o}lder
regularity ensures that $u_{q}$ is uniformly bounded in $C^{1,\beta}%
(\overline{\Omega})$ (with respect to $q$) for some $0<\beta<1.$ We conclude,
as in the proof of Lemma \ref{IntegUq}, that $u_{q}$ converges in
$C^{1}\left(  \overline{\Omega}\right)  $ to a positive weak solution $u\in
C^{1}(\overline{\Omega})\cap W_{0}^{1,p}(\Omega)$ of the eigenvalue problem
(\ref{peigen}), when $q\rightarrow p.$ Thus, $u=ke_{p}$ for some $k>0.$ But,
according to (\ref{limuq}) $k=\theta_{p},$ implying that the limit function is
always $\theta_{p}e_{p}$ (that is, it does not depend on subsequences).
Therefore, $\lim_{q\rightarrow p}u_{q}=\theta_{p}e_{p}$ in $C^{1}%
(\overline{\Omega}).$
\end{proof}

\begin{remark}
The estimate
\[
(\lambda_{p}\left\Vert \xi_{p}\right\Vert _{\infty})^{-1}\leq\liminf
\limits_{q\rightarrow p^{-}}\left\Vert u_{q}\right\Vert _{q}^{q}%
\]
where $\xi_{p}$ is the first eigenfunction normalized by the $W_{0}^{1,p}$
norm, was proved in \cite{Anello}. Since $\liminf\limits_{q\rightarrow p^{-}%
}\left\Vert u_{q}\right\Vert _{q}^{q}=\left(  \theta_{p}\left\Vert
e_{p}\right\Vert _{p}\right)  ^{p}$ (as consequence of Theorem \ref{mainTheo})
and $(\lambda_{p}\left\Vert \xi_{p}\right\Vert _{\infty})^{-1}=\left\Vert
e_{p}\right\Vert _{p}^{p}<\left(  \theta_{p}\left\Vert e_{p}\right\Vert
_{p}\right)  ^{p}$, we see this estimate is not sharp.
\end{remark}

\section{Applications}

A consequence of Theorem \ref{mainTheo} is the differentiability of the
function $q\rightarrow\lambda_{q}$ at $q=p,$ where
\[
\lambda_{q}=\min\left\{  \mathcal{R}_{q}(u):u\in W_{0}^{1,p}\left(
\Omega\right)  \setminus\{0\}\right\}
\]
and $\mathcal{R}_{q}(u):=(\left\Vert \nabla u\right\Vert _{p}/\left\Vert
u\right\Vert _{q})^{p}$ is the Rayleigh quotient associated with the immersion
$W_{0}^{1,p}(\Omega)\hookrightarrow L^{q}(\Omega),$ which is compact if
$q\in\lbrack1,p^{\ast}).$

\begin{corollary}
\label{ddiffcp}The application $q\in\lbrack1,p^{\ast})\rightarrow\lambda_{q}$
is differentiable at $q=p$ and
\begin{equation}
\frac{d}{dq}\left[  \lambda_{q}\right]  _{q=p}=\lambda_{p}\ln(\theta
_{p}\left\Vert e_{p}\right\Vert _{p}). \label{diffcp}%
\end{equation}

\end{corollary}

\begin{proof}
We recall that for each $q\in\lbrack1,p)\cup(p,p^{\ast})$ the function
\[
v_{q}=\left(  \frac{\lambda_{q}}{\lambda_{p}}\right)  ^{\frac{1}{q-p}}w_{q}%
\]
is a positive weak solution of the resonant Lane-Emden problem (\ref{lambdap}%
), where $w_{q}\in W_{0}^{1,p}\left(  \Omega\right)  \cap C^{1,\alpha}\left(
\overline{\Omega}\right)  $ satisfies%
\[
\left\Vert w_{q}\right\Vert _{q}=1\text{ \ and \ }\mathcal{R}_{q}%
(w_{q})=\lambda_{q}.
\]

Thus,
\[
\lim_{q\rightarrow p}\left\Vert v_{q}\right\Vert _{q}=\lim_{q\rightarrow
p}\left(  \frac{\lambda_{q}}{\lambda_{p}}\right)  ^{\frac{1}{q-p}}=\exp\left(
\lim_{q\rightarrow p}\frac{\ln\lambda_{q}-\ln\lambda_{p}}{q-p}\right)  .
\]
On the other hand, it follows from Theorem \ref{mainTheo} that
\[
\lim_{q\rightarrow p}\left\Vert v_{q}\right\Vert _{q}=\theta_{p}\left\Vert
e_{p}\right\Vert _{p}.
\]

Therefore,
\[
\lim_{q\rightarrow p}\frac{\ln\lambda_{q}-\ln\lambda_{p}}{q-p}=\ln(\theta
_{p}\left\Vert e_{p}\right\Vert _{p})
\]
what means that $\ln\lambda_{q}$ is differentiable at $q=p$ and $\frac{d}%
{dq}\left[  \ln\lambda_{q}\right]  _{q=p}=\ln(\theta_{p}\left\Vert
e_{p}\right\Vert _{p}).$ But this is equivalent to differentiability of
$\lambda_{q}$ at $q=p$ with $\frac{d}{dq}\left[  \lambda_{q}\right]  _{q=p}$
given by (\ref{diffcp}).
\end{proof}

Another consequence of Theorem \ref{mainTheo} is the complete description, in
the $C^{1}(\overline{\Omega})$ space, of the asymptotic behavior for the
positive solutions of the non-resonant problem ($0<\lambda\neq\lambda_{p}$):
\begin{equation}
\left\{
\begin{array}
[c]{rcll}%
-\Delta_{p}u & = & \lambda\left\vert u\right\vert ^{q-2}u & \text{in }\Omega\\
u & = & 0 & \text{on }\partial\Omega.
\end{array}
\right.  \label{LE}%
\end{equation}

\begin{corollary}
Let $\left\{  u_{\lambda,q}\right\}  _{q\in\lbrack1,p)\cup(p,p^{\ast})}$ be a
family of positive solutions of (\ref{LE}). Then%
\[
\lim\limits_{q\rightarrow p^{-}}\left\Vert u_{\lambda,q}\right\Vert _{C^{1}%
}=\left\{
\begin{array}
[c]{ccc}%
0 & \text{if} & \lambda<\lambda_{p}\\
\infty & \text{if} & \lambda>\lambda_{p}%
\end{array}
\right.  \qquad\text{and}\qquad\lim\limits_{q\rightarrow p^{+}}\left\Vert
u_{\lambda,q}\right\Vert _{C^{1}}=\left\{
\begin{array}
[c]{lll}%
\infty & \text{if} & \lambda<\lambda_{p}\\
0 & \text{if} & \lambda>\lambda_{p}.
\end{array}
\right.
\]

\end{corollary}

\begin{proof}
It follows from (\ref{anymmu}) and Lemma \ref{A} that%
\[
\left\Vert u_{\lambda,q}\right\Vert _{C^{1}}=\left(  \frac{\lambda}%
{\lambda_{p}}\right)  ^{\frac{1}{p-q}}\left\Vert u_{q}\right\Vert _{C^{1}}%
\geq\left(  \frac{\lambda}{\lambda_{p}}\right)  ^{\frac{1}{p-q}}\left\Vert
u_{q}\right\Vert _{\infty}\geq\left(  \frac{\lambda}{\lambda_{p}}\right)
^{\frac{1}{p-q}}A,
\]
for some positive constant $A$ which does not depend on $q$ close to $p.$
Thus, $\left\Vert u_{\lambda,q}\right\Vert _{C^{1}}\rightarrow\infty$ when
$q\rightarrow p^{-}$ and $\lambda>\lambda_{p}$ or when $q\rightarrow p^{+}$
and $\lambda<\lambda_{p}.$

Since $u_{q}$ is uniformly bounded in $C^{1,\beta}(\overline{\Omega})$ with
respect to $q$ close to $p$, the continuity of the immersion $C^{1,\beta
}(\overline{\Omega})\hookrightarrow C^{1}(\overline{\Omega})$ implies that
$\left\Vert u_{q}\right\Vert _{C^{1}}\leq K$ for some positive constant $K$
that does not depend on $q$ close to $p.$ Hence, when $q\rightarrow p^{-}$ and
$\lambda<\lambda_{p}$ or when $q\rightarrow p^{+}$ and $\lambda>\lambda_{p}$,
we have%
\[
\left\Vert u_{\lambda,q}\right\Vert _{C^{1}}=\left(  \frac{\lambda}%
{\lambda_{p}}\right)  ^{\frac{1}{p-q}}\left\Vert u_{q}\right\Vert _{C^{1}}\leq
K\left(  \frac{\lambda}{\lambda_{p}}\right)  ^{\frac{1}{p-q}}\rightarrow0.
\]

\end{proof}

Our results generalize those in \cite{Anello} and in \cite{GP} to $C^{1}$
norm. Note that in the super-linear case, our results are really more general
than those in \cite{GP} since they do apply to arbitrary families of positive
solutions and not only for ground states.

A third consequence of Theorem \ref{mainTheo} is that it provides a
theoretical method for obtaining approximations for first eigenpairs of the
$p$-Laplacian by solving a non-resonant problem (\ref{LE}) with $\lambda>0$
arbitrary and $q$ close to $p.$ In fact, we have the following corollary.

\begin{corollary}
\label{method}For $1\leq s\leq\infty$ and $\lambda>0$ fixed let $U_{\lambda
,q}:=\dfrac{u_{\lambda,q}}{\left\Vert u_{\lambda,q}\right\Vert _{s}}$ and
$\mu_{\lambda,q}:=\lambda\left\Vert u_{\lambda,q}\right\Vert _{s}^{q-p}.$
Then, as $q\rightarrow p:$
\[
\mu_{\lambda,q}\rightarrow\lambda_{p}\text{ \ \ and \ \ }U_{q}\rightarrow
\dfrac{e_{p}}{\left\Vert e_{p}\right\Vert _{s}}\text{ \ in \ }C^{1}%
(\overline{\Omega}).
\]

\end{corollary}

\begin{proof}
A simple scaling argument shows that
\begin{equation}
u_{\lambda,q}:=\left(  \frac{\lambda}{\lambda_{p}}\right)  ^{\frac{1}{p-q}%
}u_{q}, \label{anymmu}%
\end{equation}
where $u_{q}$ is a positive solution of the resonant Lane-Emden problem
(\ref{lambdap}). It follows from this that
\[
U_{\lambda,q}=\frac{u_{q}}{\left\Vert u_{q}\right\Vert _{s}}\text{ \ \ and
\ \ }\mu_{\lambda,q}:=\lambda\left\Vert u_{\lambda,q}\right\Vert _{s}%
^{q-p}=\lambda\frac{\lambda_{p}}{\lambda}\left\Vert u_{q}\right\Vert
_{s}^{q-p}=\lambda_{p}\left\Vert u_{q}\right\Vert _{s}^{q-p}.
\]

But, since $u_{q}\rightarrow\theta_{p}e_{p}$ in $C^{1}(\overline{\Omega})$ as
$q\rightarrow p,$ it follows that $U_{\lambda,q}=\frac{u_{q}}{\left\Vert
u_{q}\right\Vert _{s}}\rightarrow\frac{e_{p}}{\left\Vert e_{p}\right\Vert
_{s}}$ in $C^{1}(\overline{\Omega}).$ Moreover,
\[
\left\{
\begin{array}
[c]{rcll}%
-\Delta_{p}U_{\lambda,q} & = & \mu_{\lambda,q}U_{\lambda,q}^{q-1} & \text{in
}\Omega,\\
U_{\lambda,q} & = & 0 & \text{on }\partial\Omega,
\end{array}
\right.
\]
implies that
\[
\mu_{\lambda,q}=\frac{\int_{\Omega}\left\vert \nabla U_{\lambda,q}\right\vert
^{p}dx}{\int_{\Omega}\left\vert U_{\lambda,q}\right\vert ^{q}dx}%
\rightarrow\frac{\int_{\Omega}\left\vert \nabla e_{p}\right\vert ^{p}dx}%
{\int_{\Omega}\left\vert e_{p}\right\vert ^{p}dx}=\lambda_{p}%
\]
as $q\rightarrow p.$
\end{proof}

Corollary \ref{method} provides a method for obtaining numerical
approximations of the first eigenpair $(\lambda_{p},\frac{e_{p}}{\left\Vert
e_{p}\right\Vert _{s}}).$ In fact, in a first step one can compute a numerical
solution of problem (\ref{LE}) with $q$ close to $p$ and hence, after $L^{s}%
$-normalization one obtains approximations for $\lambda_{p}$ and $\frac{e_{p}%
}{\left\Vert e_{p}\right\Vert _{s}}$ simultaneously.

Of course, a numerical solution of the nonlinear problem (\ref{LE}), for some
$\lambda>0$ fixed, is easier to obtain than directly compute the first
eigenpair of the $p$-Laplacian (by solving the corresponding eigenvalue
problem). As previously mentioned, the advantage here is that $\lambda$ can be
chosen arbitrarily in computational implementations of (\ref{LE}) and does not
need to be close to $\lambda_{p}.$

We emphasize that this approach is well-supported by our results in this
paper, especially in the super-linear case, since it does apply to any family
of positive solutions and not only to ground state. Note that such a
restriction would discourage the application of this method since it would be
necessary to prove that a numerical solution of the super-linear problem is in
fact a ground state.

A similar approach was recently used in \cite{BBEM}, where the iterative sub-
and super-solution method was applied to compute the positive solutions of the
sub-linear problem.

Now, we show that the quotient
\[
\Lambda_{q}:=\lambda\dfrac{\left\Vert u_{\lambda,q}\right\Vert _{q}^{q}%
}{\left\Vert u_{\lambda,q}\right\Vert _{p}^{p}}=\lambda\frac{(\frac{\lambda
}{\lambda_{p}})^{\frac{q}{p-q}}}{(\frac{\lambda}{\lambda_{p}})^{\frac{p}{p-q}%
}}\dfrac{\left\Vert u_{q}\right\Vert _{q}^{q}}{\left\Vert u_{q}\right\Vert
_{p}^{p}}=\lambda_{p}\dfrac{\left\Vert u_{q}\right\Vert _{q}^{q}}{\left\Vert
u_{q}\right\Vert _{p}^{p}},
\]
also converges to $\lambda_{p}$ as $q\rightarrow p.$ Moreover, we estimate the
convergence order in the approximation of $\Lambda_{q}$ by $\mu_{\lambda,q}.$

\begin{proposition}
\textit{There holds:}

\begin{enumerate}
\item[(i)] $\lambda_{p}\leqslant\Lambda_{q}.$

\item[(ii)] $\lim\limits_{q\rightarrow p}\Lambda_{q}=\lambda_{p}.$

\item[(iii)] If $q$ is sufficiently close to $p,$ then%
\[
\left\vert \dfrac{\Lambda_{q}}{\mu_{\lambda,q}}-1\right\vert \leq K\left\vert
q-p\right\vert
\]
for some positive constant $K$ which does not depends on $q.$
\end{enumerate}
\end{proposition}

\begin{proof}
Since $\lim\limits_{q\rightarrow p}\left\Vert U_{q}\right\Vert _{q}^{q}%
=\lim\limits_{q\rightarrow p}\left\Vert U_{q}\right\Vert _{p}^{p}=\left\Vert
e_{p}\right\Vert _{p}^{p}$, we have
\[
\lambda_{p}\leq\frac{\left\Vert \nabla u_{\lambda,q}\right\Vert _{p}^{p}%
}{\left\Vert u_{\lambda,q}\right\Vert _{p}^{p}}=\dfrac{\lambda\left\Vert
u_{\lambda,q}\right\Vert _{q}^{q}}{\left\Vert u_{\lambda,q}\right\Vert
_{p}^{p}}=\Lambda_{q}=\lambda\left\Vert u_{\lambda,q}\right\Vert _{\infty
}^{q-p}\dfrac{\left\Vert U_{q}\right\Vert _{q}^{q}}{\left\Vert U_{q}%
\right\Vert _{p}^{p}}=\mu_{\lambda,q}\dfrac{\left\Vert U_{q}\right\Vert
_{q}^{q}}{\left\Vert U_{q}\right\Vert _{p}^{p}}%
\]
proving \textit{(i)} and \textit{(ii)}, since
\[
\lambda_{p}\leq\lim_{q\rightarrow p}\Lambda_{q}=\lim_{q\rightarrow p}%
\mu_{\lambda,q}\dfrac{\left\Vert U_{q}\right\Vert _{q}^{q}}{\left\Vert
U_{q}\right\Vert _{p}^{p}}=(\lim_{q\rightarrow p}\mu_{\lambda,q})\left(
\lim\limits_{q\rightarrow p^{-}}\frac{\left\Vert U_{q}\right\Vert _{q}^{q}%
}{\left\Vert U_{q}\right\Vert _{p}^{p}}\right)  =\lambda_{p}.
\]

Now, \textit{(iii)} follows since
\begin{align*}
\left\vert \frac{\Lambda_{q}}{\mu_{\lambda,q}}-1\right\vert  &  =\left\vert
\frac{\left\Vert U_{q}\right\Vert _{q}^{q}}{\left\Vert U_{q}\right\Vert
_{p}^{p}}-1\right\vert \\
&  =\frac{1}{\left\Vert U_{q}\right\Vert _{p}^{p}}\left\vert \int_{\Omega
}\left(  U_{q}^{q}-U_{q}^{p}\right)  dx\right\vert \leq\frac{1}{\left\Vert
U_{q}\right\Vert _{p}^{p}}\int_{\Omega}\max_{0\leqslant t\leqslant1}\left\vert
t^{q}-t^{p}\right\vert dx\leq K\left\vert p-q\right\vert ,
\end{align*}
taking into account that $\left\Vert U_{q}\right\Vert _{p}^{p}\rightarrow
\left\Vert e_{p}\right\Vert _{p}^{p}$ as $q\rightarrow p.$
\end{proof}

\begin{remark}
Since $\mu_{\lambda,q}\leq\lambda_{p}$ in the sub-linear case (see
\cite{BBEM}), the rate of convergence of both $\mu_{\lambda,q}$ and
$\Lambda_{q}$ to $\lambda_{p}$ is at least $O(p-q).$
\end{remark}

\section{An explicit estimate involving $L^{\infty}$ and $L^{1}$ norms of
$u_{\lambda,q}$}

Still regarding the approximation method of Corollary \ref{method} we remark
that $L^{s}$-normalization ($1\leq s<\infty$) is linked to robustness and
stability in a computational approach. In this sense, explicit estimates
involving $L^{\infty}$ and $L^{s}$ norms of $u_{\lambda,q}$ in terms of the
parameters $p,$ $N,$ $q$ and $\lambda$ are also useful.

We end this paper by presenting an explicit estimate involving $L^{\infty}$
and $L^{s}$ norms of $u_{\lambda,q}.$ Its proof is inspired by arguments based
on level set techniques developed in \cite{Ladyz} (see also \cite[Theor.
2]{Bandle}). We need the following lemma.

\begin{lemma}
\label{D}Let $D\subset\mathbb{R}^{N}$ be a bounded and smooth domain. Then,
\begin{equation}
\int_{D}|u|^{p}dx\leq\frac{|D|^{\frac{p}{N}}}{C_{N,p}}\int_{D}|\nabla
u|^{p}\,dx \label{cp*}%
\end{equation}
for all $u\in W_{0}^{1,p}(D)\setminus\{0\}$, where $C_{N,p}$ is given by
\begin{equation}
C_{N,p}=N\omega_{N}^{\frac{p}{N}}\left(  \frac{p}{p-1}\right)  ^{p-1}
\label{CNp}%
\end{equation}
and $\omega_{N}$ is the volume of the unit ball in $\mathbb{R}^{N}.$
\end{lemma}

\begin{proof}
This result was proved in \cite[Corollary 8]{BE}. For completeness we outline
here its proof. Let $\lambda_{p}(D)$ denote the first eigenvalue of the
$p$-Laplacian in $D.$ We need only to verify that $C_{N,p}\left\vert
D\right\vert ^{-\frac{p}{N}}$ is a lower bound for $\lambda_{p}(D).$ It
follows from (\ref{lblp}) that $\left\Vert \phi_{p}\right\Vert _{\infty}%
^{1-p}$ is also a lower bound for $\lambda_{p}(D),$ where $\phi_{p}$ is the
$p$-torsion function of $D$ (as in (\ref{torsion}) with $\Omega$ replaced by
$D$)$.$ But, as consequence of Talenti's comparison principle (see
\cite{Talenti}) we have that $\phi_{p}^{\ast}\leq\Phi_{p}$ in $D^{\ast}$
where: $D^{\ast}$ is the ball in $\mathbb{R}^{N}$ centered at the origin and
such that $\left\vert D^{\ast}\right\vert =\left\vert D\right\vert ,$
$\phi_{p}^{\ast}$ is the Schwarz symmetrization of $\phi_{p}$ (see
\cite{Kawohl}) and $\Phi_{p}$ is the $p$-torsion function of $D^{\ast}.$
Therefore, since $\left\Vert \phi_{p}\right\Vert _{\infty}=\left\Vert \phi
_{p}^{\ast}\right\Vert _{\infty}$ it follows that $\left\Vert \Phi
_{p}\right\Vert _{\infty}^{1-p}$ is a lower bound for $\lambda_{p}(D).$ The
torsion function $\Phi_{p}$ is a radially decreasing function that can be
explicitly computed:%
\[
\Phi_{p}(\left\vert x\right\vert )=\frac{p-1}{p}N^{-\frac{1}{p-1}}\left(
R_{D}^{\frac{p}{p-1}}-\left\vert x\right\vert ^{\frac{p}{p-1}}\right)  ,\text{
\ \ }\left\vert x\right\vert \leq R_{D}%
\]
where $R_{D}:=(\left\vert D\right\vert /\omega_{N})^{1/N}$ is the radius of
the ball $D^{\ast}.$ Thus, one has
\[
\left\Vert \Phi_{p}\right\Vert _{\infty}^{1-p}=\left(  \Phi_{p}(0)\right)
^{1-p}=\left(  \frac{p-1}{p}N^{-\frac{1}{p-1}}R_{D}^{\frac{p}{p-1}}\right)
^{1-p}=C_{N,p}\left\vert D\right\vert ^{-\frac{p}{N}}.
\]

\end{proof}

\begin{theorem}
\label{linfty}Let $u_{\lambda,q}\in W_{0}^{1,p}\left(  \Omega\right)  $ be a
positive weak solution of the resonant Lane-Emden
\[
\left\{
\begin{array}
[c]{rcll}%
-\Delta_{p}u & = & \lambda\left\vert u\right\vert ^{q-2}u & \text{in }%
\Omega,\\
u & = & 0 & \text{on }\partial\Omega,
\end{array}
\right.
\]
where $\lambda>0$ and $1\leq q<p(\frac{N+1}{N}).$ Then, one has
\begin{equation}
\left\Vert u_{\lambda,q}\right\Vert _{\infty}\leq\left(  \lambda^{\frac{N}{p}%
}K_{N,p}\left\Vert u_{\lambda,q}\right\Vert _{1}\right)  ^{\frac{p}{p+N(p-q)}%
}, \label{x4}%
\end{equation}
where
\begin{equation}
K_{N,p}:=C_{N,p}^{-\frac{N}{p}}\left(  \frac{p+N(p-1)}{p}\right)
^{\frac{p+N(p-1)}{p}} \label{Knplamb}%
\end{equation}
and $C_{N,p}$ is defined by (\ref{CNp}).
\end{theorem}

\begin{proof}
Let us denote $u_{\lambda,q}$ simply by $u.$ For each $0<t<\left\Vert
u\right\Vert _{\infty},$ define
\[
A_{t}:=\left\{  x\in\Omega:u>t\right\}  .
\]

The function
\[
\left(  u-t\right)  ^{+}=\max\left\{  u-t,0\right\}  =\left\{
\begin{array}
[c]{ll}%
u-t, & \text{if \ }u>t,\\
0, & \text{if \ }u\leq t,
\end{array}
\right.
\]
belongs to $W_{0}^{1,p}\left(  \Omega\right)  $. Classical results
\cite{DiBenedetto, Lieberman, Tolks} guarantee that $u_{\lambda,q}\in
C^{1,\alpha}\left(  \overline{\Omega}\right)  $ for some $0<\alpha<1.$
Therefore, we have%
\begin{equation}
\int_{A_{t}}\left\vert \nabla u\right\vert ^{p}dx=\lambda\int_{A_{t}}%
u^{q-1}\left(  u-t\right)  dx\leq\lambda\left\Vert u\right\Vert _{\infty
}^{q-1}\int_{A_{t}}\left(  u-t\right)  dx. \label{x1}%
\end{equation}
(Note that $A_{t}$ is open and therefore $\nabla(u-t)^{+}=\nabla u$ in
$A_{t}.$)

Now, we estimate $\int_{A_{t}}\left\vert \nabla u\right\vert ^{p}dx$ from
below. For this, we apply H\"{o}lder's inequality and (\ref{cp*}) from Lemma
\ref{D} with $D=A_{t}$ to obtain%
\[
\left(  \int_{A_{t}}\left(  u-t\right)  dx\right)  ^{p}\leq|A_{t}|^{p-1}%
\int_{A_{t}}\left(  u-t\right)  ^{p}dx\leq\frac{|A_{t}|^{p-1}|A_{t}|^{\frac
{p}{N}}}{C_{N,p}}\int_{A_{t}}|\nabla u|^{p}\,dx.
\]

Thus,
\[
C_{N,p}|A_{t}|^{-\frac{p}{N}+1-p}\left(  \int_{A_{t}}\left(  u-t\right)
dx\right)  ^{p}\leq\int_{A_{t}}\left\vert \nabla u\right\vert ^{p}dx
\]
what yields, taking into account (\ref{x1}),%
\[
C_{N,p}|A_{t}|^{-\frac{p}{N}+1-p}\left(  \int_{A_{t}}\left(  u-t\right)
dx\right)  ^{p}\leq\lambda\left\Vert u\right\Vert _{\infty}^{q-1}\int_{A_{t}%
}\left(  u-t\right)  dx
\]
So we have%
\[
\left(  \int_{A_{t}}\left(  u-t\right)  dx\right)  ^{p-1}\leq\frac{\lambda
}{C_{N,p}}\left\Vert u\right\Vert _{\infty}^{q-1}\left\vert A_{t}\right\vert
^{\frac{p+N(p-1)}{N}}%
\]

This last inequality can be rewritten as%
\begin{equation}
\left(  \int_{A_{t}}\left(  u-t\right)  dx\right)  ^{\frac{N(p-1)}{p+N(p-1)}%
}\leq\left(  \frac{\lambda}{C_{N,p}}\left\Vert u\right\Vert _{\infty}%
^{q-1}\right)  ^{\frac{N}{p+N(p-1)}}\left\vert A_{t}\right\vert . \label{x2}%
\end{equation}

By defining
\[
f(t):=\int_{A_{t}}\left(  u-t\right)  dx,
\]
it follows from Cavalieri's Principle that%
\[
f(t)=\int_{t}^{\infty}\left\vert A_{s}\right\vert ds
\]
and therefore $f^{\prime}\left(  t\right)  =-\left\vert A_{t}\right\vert .$
Thus, (\ref{x2}) can be rewritten as
\begin{equation}
1\leq-\left(  \frac{\lambda}{C_{N,p}}\left\Vert u\right\Vert _{\infty}%
^{q-1}\right)  ^{\frac{N}{p+N(p-1)}}f\left(  t\right)  ^{-\frac{N(p-1)}%
{p+N(p-1)}}f^{\prime}\left(  t\right)  . \label{x3}%
\end{equation}

Integration of (\ref{x3}) yields
\begin{align*}
t  &  \leq(\lambda^{\frac{N}{p}}K_{N,p})^{\frac{p}{p+N(p-1)}}\left(
\left\Vert u\right\Vert _{\infty}\right)  ^{\frac{N(q-1)}{p+N(p-1)}}\left[
f(0)^{\frac{p}{p+N(p-1)}}-f(t)^{\frac{p}{p+N(p-1)}}\right] \\
&  \leq(\lambda^{\frac{N}{p}}K_{N,p})^{\frac{p}{p+N(p-1)}}\left(  \left\Vert
u\right\Vert _{\infty}\right)  ^{\frac{N(q-1)}{p+N(p-1)}}\left(  \left\Vert
u\right\Vert _{1}\right)  ^{\frac{p}{p+N(p-1)}}%
\end{align*}
where%
\[
K_{N,p}:=C_{N,p}^{-\frac{N}{p}}\left(  \frac{p+N(p-1)}{p}\right)
^{\frac{p+N(p-1)}{p}}.
\]

Making $t\rightarrow\left\Vert u\right\Vert _{\infty}$ we obtain
\[
\left\Vert u\right\Vert _{\infty}\leq(\lambda^{\frac{N}{p}}K_{N,p})^{\frac
{p}{p+N(p-1)}}\left(  \left\Vert u\right\Vert _{\infty}\right)  ^{\frac
{N(q-1)}{p+N(p-1)}}\left(  \left\Vert u\right\Vert _{1}\right)  ^{\frac
{p}{p+N(p-1)}}%
\]
and hence, by noting that
\[
0\leq\frac{N(q-1)}{p+N(p-1)}<1\Longleftrightarrow1\leq q<p\left(  \frac
{N+1}{N}\right)
\]
we obtain%
\[
\left(  \left\Vert u\right\Vert _{\infty}\right)  ^{\frac{p+N(p-q)}{p+N(p-1)}%
}\leq\left(  \lambda^{\frac{N}{p}}K_{N,p}\left\Vert u\right\Vert _{1}\right)
^{\frac{p}{p+N(p-1)}}%
\]
which gives (\ref{x4}).
\end{proof}

A well-known Sobolev inequality could be used instead of (\ref{cp*}) to prove
a similar estimate to (\ref{x4}). However, (\ref{cp*}) does not impose any
relation between $p$ and $N$ while the Sobolev inequality requires $1<p<N.$

Let us remark that the estimate (\ref{x4}) might, in principle, be used in the
analysis of nodal solutions of the Lane-Emdem problem. In fact, it is valid in
each nodal domain and does not depend on the Lebesgue volume of it.

\end{document}